\documentclass[12pt]{article}
\usepackage{graphicx}
\usepackage{amssymb}
\usepackage{amsmath}
\usepackage{amsfonts}
\usepackage{amsthm}
\usepackage[english]{babel}
\usepackage[latin1,ansinew]{inputenc}
\usepackage{a4wide}
\usepackage{color}
\newcommand{\be}{\begin{eqnarray}}
\newcommand{\ee}{\end{eqnarray}}
\newtheorem{theo}{Theorem}%[section]
\newtheorem{lemma}{Lemma}

\newtheorem{defi}{Definition}
\newtheorem{rem}{Remark}
\newcommand{\R}{\mathbb R}
\newcommand{\C}{\mathbb C}

\newcommand{\new}{\textcolor{red}}

\newcommand{\bxi}{\boldsymbol\xi}
\newcommand{\bomega}{\boldsymbol\omega}

\begin{document}

\title{Nonlinear Stability of \\ First-Order Relativistic Viscous Hydrodynamics\\}
\author{\it Heinrich Freist\"uhler\thanks{Department of Mathematics, University of Konstanz, 
78457 Konstanz, Germany} ~ {\rm and}\ \  Matthias Sroczinski$^*$}
\maketitle
\begin{abstract}
The paper shows time-asymptotic nonlinear stability of homogenous reference states for 
a very general class of first-order descriptions of relativistic viscous hydrodynamics.
This result, sharper than the previously known mere decay of individual Fourier modes, 
thus applies to the descriptions formulated in H. Freist\"uhler and B. Temple, 
{\em Proc.\ R. Soc.\ A} {\bf 470}, 20140055 (2014),
{\em J. Math.\ Phys.}\ {\bf 59}, 063101 (2018), 
F. S. Bemfica, M. Disconzi, and J. Noronha, 
{\em Phys.\ Rev.\ D} {\bf 98},104064 (2018), 
{\em Phys. Rev. D} {\bf 100}, 104020 (2019),
and Freist\"uhler, {\em J. Math.\ Phys.}\ {\bf 61}, 033101 (2020).
\end{abstract}
\newpage
\section{Introduction}
\setcounter{equation}{0}
Various proposals have recently been made for describing 
dissipative relativistic fluid dynamics by second-order 
systms of partial differential equations in the classical 
fluid variables, i.e., the velocity and one (in the barotropic)
or two (in the non-barotropic case) thermodynamic quantities.
The stability of such descriptions has been addressed by 
showing that linear waves, i.e., Fourier modes of the linearized system,
do not grow in time. While significant, this property does not 
only miss capturing the actual nonlinearity, 
but the mere knowledge of non-growth of modes does also not 
suffice to show stability in the sense of decay to homogeneous 
reference states at controlled temporal rates even at the linear level. 
The purpose of 
this paper is to close this gap for a broad class of models
for barotropic fluids.

We study dissipative relativistic fluid dynamics as described by  models of the form 
%a system of partial differential equations 
\be\label{fieldmodel}
\frac{\partial}{\partial x^{\beta}}\left(T^{\alpha\beta}+\Delta T^{\alpha\beta}\right)=0,
\ee
were the ideal part of the stress-energy tensor is given by
\be
T^{\alpha\beta}=\theta p'(\theta)u^\alpha u^\beta+p(\theta)g^{\alpha\beta},
\ee
with $u^\alpha$ the 4-velocity and the fluid is specified by an equation of state 
$p=p(\theta)$ that gives its pressure $p$ as a function of its temperature $\theta$, with 
\begin{equation} 
	\label{ba}
	p'(\theta),p''(\theta)>0. %\quad 0<c_s=\sqrt{p'(\theta)/p''(\theta)}<1.
\end{equation}
%\goodbreak
Regarding the dissipative part, we explore tensors of the general form 
\cite{ft14,ft18}
\be
-\Delta T^{\alpha\beta}
\equiv
u^\alpha u^\beta P
+(\Pi^{\alpha\gamma} u^\beta + \Pi^{\beta\gamma}u^\alpha) Q_\gamma
+\Pi^{\alpha\beta}R 
+\Pi^{\alpha\gamma}\Pi^{\beta\delta}S_{\gamma\delta}
\label{ournewansatz}
\ee
\goodbreak
\new{where\footnote{\new{We use six minus signs so as to easily accommodate the BDN models
in the sequel.}}}
$$
P
=
\new{-}\tilde\kappa u^\gamma{\partial \theta\over \partial x^\gamma}
\new{-}\tilde\tau{\partial u^\gamma\over \partial x^\gamma},
\quad
R
=
\new{-}\tilde\omega u^\gamma{\partial \theta\over \partial x^\gamma}
\new{-}\tilde\chi{\partial u^\gamma\over \partial x^\gamma},
$$
\[
Q_\gamma
\equiv
\new{-}\tilde\nu{\partial \theta\over \partial x^\gamma}
\new{-}\tilde\mu u^\delta
{\partial u_\gamma\over \partial x^\delta},
\quad
S_{\gamma\delta}
\equiv
\tilde\eta\bigg({\partial u_\gamma\over \partial x^\delta}
+
{\partial u_\delta\over \partial x^\gamma}
-{2\over 3}g_{\gamma\delta}
{\partial u^\epsilon\over \partial x^\epsilon}\bigg),
\]
\[
\Pi^{\alpha\beta}\equiv g^{\alpha\beta}+u^\alpha u^\beta
\]
and the dissipation coefficients 
$$
\tilde\kappa,\tilde\tau,\tilde\omega,\tilde\chi,\tilde\nu,\tilde\mu,\tilde\eta
$$
are, in principle arbitrary, functions of $\theta$.
Using the natural variable 
$
\psi^\alpha=\theta^{-1}u^\alpha
$
that ranges in $\mathcal U=\{\psi \in \R^4: \psi_\alpha \psi^\alpha<0\}$, 
we express \eqref{fieldmodel} as
\begin{equation}
 	\label{cp1}
	 A^{\alpha\beta\gamma}(\psi(x))\frac{\partial\psi_\gamma}{\partial x^\beta}(x)
	 =
	 \frac\partial{\partial x^\beta }
	 \left(B^{\alpha\beta\gamma\delta}(\psi(x))\frac{\partial\psi_\gamma}{\partial x^\delta}\right),
        % \quad  x=(x_0,x_1,x_2,x_3) \in [0,\infty)\times \R^3,\\
\end{equation}
and consider solutions  $\psi:[0,\infty) \times \R^3 \to \mathcal U$ of \eqref{cp1} which satisfy
initial conditions 
\begin{equation}\label{cp2}
 \psi(0,\cdot)={^0}\psi, \quad \frac{\partial \psi}{\partial x_0}(0,\cdot)={^1}\psi.
\end{equation}
with given data ${^0}\psi,{^1}\psi:\R^3 \to \mathcal U$.

The purpose of the paper is to show nonlinear stability of homogeneous states in the following sense.

\begin{theo}
\label{nls}
Fix a number $s>5/2$ and assume that normalized versions 
$(\kappa,\omega,\nu)=\theta^{-2}(\tilde\kappa,\tilde\omega,\tilde\nu)$, 
$(\tau,\chi,\mu,\eta)=\theta^{-1}(\tilde\tau,\tilde\chi,\tilde\mu,\tilde\eta)$
of the dissipation coefficients           
%$\tilde\kappa,\tilde\tau,\tilde\omega,\tilde\chi,\tilde\nu,\tilde\mu,\tilde\eta$
satisfy  
$$
\kappa,\mu,\eta,\nu\sigma>0
\quad\text{ with }\quad 
\sigma=\chi-4\eta/3
$$
and either condition (C1), i.e,
\begin{align}
			\label{hb2}\tag{C1.1}
		 (\tau+\mu)(\nu+\omega)-\kappa\sigma-\nu\mu&>0,\\
		\label{hb3}\tag{C1.2} ((\tau+\mu)(\nu+\omega)-\kappa\sigma-\nu\mu)^2
                 -4\nu\mu\kappa\sigma&>0,\\
			\label{D31}\tag{C1.3}
			c_s^{-2}(\omega+\tau
			)-\kappa-c_s^{-4}\sigma &> 0,\\
			\nonumber
			(\kappa+c_s^{-2}\mu)(\tau+\omega+\mu- c_s^{-2}\sigma)[(\tau+\mu)(\omega+\nu)
                        -(\mu\nu+\kappa\sigma)]\qquad\qquad&\\
			\label{D32}\tag{C1.4}
			-\kappa \mu(\tau+\omega+\mu-c_s^{-2}\sigma)^2
                        -(\kappa+c_s^{-2}\mu)^2\nu\sigma &> 0,
			\end{align}
or condition (C2), i.e.,
	\begin{align}
		\label{C21}\tag{C2.1}
		\kappa\sigma=\nu\mu&<0,\quad \tau+\mu=\omega+\nu=0,\\
		\label{C22}\tag{C2.2}
		&\sigma+c_s^2\mu < 0,
	\end{align}
where $c_s=\sqrt{p'(\theta)/\theta p''(\theta)}>0$ is the speed of sound.\\ 
Then for any rest state $\bar\psi=(\bar\theta^{-1},0,0,0), \bar\theta>0$, 
there exist constants $\delta,C>0$ such that the following holds for all function pairs 
${^0}\psi,{^1}\psi: \R^3 \to \mathcal U$ with ${^0}\psi-\bar \psi  \in H^{s+1}\cap L^1$, 
${^1}\psi \in H^{s}\cap L^1$:\\
\goodbreak
If 
$$ %\begin{equation}\label{datasmall}
 \|{^0}\psi-\bar \psi\|_{H^{s+1}},\|{^1}\psi\|_{H^s},\|{^0}\psi-\bar \psi\|_{L^1}, \|{^1}\psi\|_{L^1} <\delta,
$$ %\end{equation}
then there exists a unique global solution $\psi$ of \eqref{cp1}, \eqref{cp2} satisfying 
$\psi-\bar \psi \in C^0([0,\infty),H^{s+1})\cap C^1([0,\infty),H^{s})$ and, for all $t \in [0,\infty)$, 
	\begin{align*}
		\|\psi(t)-\bar \psi\|_{H^s}+\|\psi_t(t)\|_{H^{s-1}}  &\le C(1+t)^{-\frac{3}{4}}(\|{^0}\psi-\bar \psi\|_{H^s}+\|{^1}\psi\|_{H^{s-1}}
+\|{^0}\psi-\bar \psi\|_{L^1}+\|{^1}\psi\|_{L^1}).
	\end{align*}
\end{theo} 

%\vskip -0.5mm
I.\ e., at least for initial data that are sufficiently small perturbations 
of the homogeneous reference state $\bar\psi$, a unique solution 
to the nonlinear Cauchy problem exists globally in time 
and decays time-asymptotically to $\bar\psi$ at the rate $t^{-3/4}$.  
\par\vskip .5mm
In the statement of the theorem, the notations $C^k,H^l,$ and $L^1$ mean the usual spaces of functions 
that are $k$ times continuously differentiable, square integrable together with their
derivatives of up to $l$th order, or integrable, respectively, and $\Vert.\Vert_{H^l},
\Vert .\Vert_{L^1}$ are the natural norms in $H^l$ and $L^1$. 
%For the argumentation in this paper 
%we do not need to enter any details of these function spaces. 
We derive the assertion by appealing to a result   
of the second author \cite{sro23} which reduces global existence and decay of solutions in these
spaces for a general class of systems of the form \eqref{cp1} to two algebraic criteria, the hyperbolicity 
condition $(H)_B$ and the dissipativity condition $(D)$ which we detail in Secs.\ 2 and 3. 
Our concrete task here is thus the verification of these conditions under the assumptions (C1) 
or (C2). 

%\par\vskip .5mm
Theorem 1 treats models introduced in \cite{ft14,ft18,bem18,bem19} from a common perspective
regarding specific aspects of their dissipativity. While the decay of Fourier modes as such had been 
considered in these papers, we here check the refined criterion (D) of \cite{FS} that, differently from 
the hyperbolic-parabolic situation of \cite{kawa}, is needed in the present hyperbolic-hyperbolic 
context in addition to mere decay of modes 
in order to ensure asymptotic stability in the sense of decay of solutions in appropriate function spaces.\\
For the models considered in \cite{bem19}, criterion (D) indeed again amounts to the same inequalities,
namely (C.1), as are stated in that paper for the most relevant case $\nu=\mu$.\footnote{Inequalities
(C1.3), (C1.4) only look slightly different from (10a) and (10b) in \cite{bem19} since {\it some} 
coefficients scale with $c_s^2$, thus $\chi_1=c_s^2\kappa$ and $\chi_3=c_s^2\omega$.} But only our checking (D) here 
confirms the decay in Sobolev spaces we report in Theorem 1.\\
For two prototypical cases, the assertion of Theorem 1 has been 
established before, in \cite{sro18,sro23}.

%\par\vskip .5mm
The theorem  generalizes to non-quiescent rest states ($\bar u\neq (1,0,0,0)$). While the mere decay of Fourier 
modes transfers from the rest frame (under natural causality assumptions, cf.\ Sec.\ 4 and \cite{bem22}),  
corresponding analysis of criterion (D) is left to a future publication. 

\begin{rem}
We note that Theorem 1 distinguishes two situations: the positive case ((C2), $B_{\parallel}>0$) and the negative case ((C1), $B_{\parallel}<0$). In the positive case,
the model is of Hughes-Kato-Marsden type \cite{Hug77,ft14}, in the negative case it is not.
\end{rem}

%\par\vskip .5mm
In the following, we use the matrix notation
$$ 
B^{\beta\delta}=( B^{\alpha\beta\gamma\delta})_{0\le \alpha,\gamma \le 3},
\quad  
A^{\beta}=( A^{\alpha\beta\gamma})_{0 \le \alpha,\gamma \le 3}
$$
and the Fourier symbols
\begin{align*}
	 B(\psi,\bxi)&= B^{jk}(\psi)\xi_j\xi_k, \quad
	 C(\psi,\bxi)=(B^{0j}(\psi)+ B^{j0}(\psi))\xi_j,\\
	 A(\psi,\bxi)&= A^j(\psi)\xi_j,~~\bxi=(\xi_1,\xi_2,\xi_3) \in \R^3,~\xi=|\bxi|.
\end{align*}
\goodbreak
\newpage
\section{Hyperbolicity}
\setcounter{equation}0
By a classical result of Taylor \cite{T91} the Cauchy problem \eqref{cp1}, \eqref{cp2} is locally well-posed 
if the differential operator 
\begin{equation}\label{diffopB}
\psi\mapsto
B^{\alpha\beta\gamma\delta}(\psi(x))\frac{\partial^2\psi_\gamma}
{\partial x^\beta\partial x^\delta}
\end{equation}
%tensor $B^{\alpha\beta\gamma\delta}$ 
is (second-order) hyperbolic in the following sense:

\begin{enumerate}
	\item[(H$_B$)] (a) $B^{00}$ is negative definite,\\
	(b) with $\check B^{\beta\delta}$ defined through
	$$
	\check B^{\beta\delta}= (-B^{00})^{-1/2}B^{\beta\delta}(-B^{00})^{-1/2},
	$$
	the matrix family
	$$
	i\mathcal B(\psi,\bomega)=i\begin{pmatrix}
		0&I_4\\
		\check B(\psi,\bomega)&i\check C(\psi,\bomega)
	\end{pmatrix},
	\quad (\psi,\bomega)\in\mathcal U\times S^{2},
	$$
	permits a symbolic symmetrizer $\mathcal S$.
\end{enumerate}

Recall that for 
%\begin{defi}
%	\label{symsym}
	$M \in C^\infty(\mathcal U \times S^{d-1},\C^{n \times n})$, a symbolic symmetrizer is a 
        smooth mapping $\Sigma \in C^\infty(\mathcal U \times S^{d-1},\C^{n \times n})$, 
        bounded as well as all its derivatives,  such that for some $c>0$ and all 
        $(\psi,\bomega) \in \mathcal U \times S^{d-1}$
	$$\Sigma(\psi,\bomega)=\Sigma(\psi,\bomega)^* \ge cI_n,
        ~~\Sigma(\psi,\bomega)M(\psi,\bomega)=(\Sigma(\psi,\bomega)M(\psi,\omega))^*.$$
%\end{defi}
Clearly, if a matrix family $M$ admits a symbolic symmetrizer all eigenvalues of $M(\psi,\bomega)$ 
are real and semi-simple. On the other hand a matrix family $M$ admits a symbolic symmetrizer if the 
latter holds and additionally the multiplicities of the eigenvalues of 
$M(\psi,\bomega)$ are constant on $\mathcal U \times S^{d-1}$. This motivates the following notion.

%\begin{defi}
\textit{Definition. We call $B^{\alpha\beta\gamma\delta}$ or, more precisely, the differential operator 
\eqref{diffopB}, 
\emph{semi-strictly hyperbolic} if 
(a) holds and for all $(\psi,\bomega) \in \mathcal U\times S^{2}$ the eigenvalues of 
$i\mathcal B(\psi,\bomega)$ as defined in (b) are real semi-simple with multiplicities 
independent of $(\psi,\bomega)$.}

Since we will show that $B^{\alpha\beta\gamma\delta}$ is hyperbolic if and only if it is semi-strictly 
hyperbolic,  it is sufficient to show (H$_B$) in the rest frame due to Lorentz invariance. 
In the following all matrices are evaluated at $\bar \psi$ without further indication.

First note that the rest-frame representations are explicitly given as
$$
B^{00}=-\begin{pmatrix} \kappa &0 \\ 0 & \mu I_3 \end{pmatrix}
$$
and
$$
B(\bxi)=
  -\begin{pmatrix} \nu \xi^2 & 0 \\0 & (\chi-\frac13 \eta)\xi\xi^t-\eta\xi^2 I_3 \end{pmatrix},
\quad C(\bxi)=
  -\begin{pmatrix} 0& (\tau+\mu)\bxi^t\\
  (\omega+\nu)\bxi & 0\end{pmatrix}.
$$
It is straightforward to see that for any $\bxi \in \R^3$ the matrices 
$B^{00}, B(\bxi), C(\bxi)$ all decompose in the sense of linear operators as 
$B^{00}=B^{00}_\parallel\oplus B^{00}_\perp$,  $B(\bxi)=B_\parallel(\xi) \otimes B_\perp(\xi)$, 
$C(\xi)=C_\parallel(\xi) \oplus C_\perp(\xi)$ with respect to the orthogonal decomposition 
$\C^4=(\C \times\C\bxi)\oplus(\{0\} \times \{\bxi\}^{\perp})$, 
where
\begin{align}\label{BBC}
B^{00}_\parallel&=-\begin{pmatrix}\kappa & 0\\0 & \mu\end{pmatrix},&B^{00}_\perp&=-\mu I_2,\notag\\
B_{\parallel}(\xi)&=-\xi^2\begin{pmatrix} \nu & 0\\0 & \sigma \end{pmatrix},&B_\perp(\xi)&=\xi^2\eta I_2,\\
C_\parallel(\xi)&=-\xi\begin{pmatrix} 0 & \tau+\mu\\ \omega+\nu & 0\end{pmatrix},&C_\perp(\xi)&=0,\notag
	\end{align}
so that we can treat hyperbolicity on 
$\C \times \C\bxi$ 
and
$\{0\} \times \{\bxi\}^\perp$  
separately. We also just write $B_\parallel$ instead of 
$B_\parallel(1)$, etc., if there is no concern for confusion.
	.
\begin{lemma}
	\label{hb}
$B^{\alpha\beta\gamma\delta}$ satisfies (H$_B$) if and only if 
$\kappa,\mu,\eta>0$ and either 
\begin{itemize}
 \item[(i)] $\nu\sigma>0$ and \eqref{hb2} and \eqref{hb3}, or
 \item[(ii)] $\nu\sigma>0$ and \eqref{C21}, or 
 \item[(iii)] $\nu=\sigma=0 ~\text{~and~}(\tau+\mu)(\omega+\nu)>0$
\end{itemize}
are true.
\end{lemma}

\begin{proof}
	Trivially, $-B^{00}<0$ is equivalent to $\kappa, \mu >0$; we assume this to be the case.  
Next, we treat (H$_B$) (b), separately on the spaces $\C \times\C\bxi$ and 
$\{0\} \times \{\bxi\}^\perp$. We first find that 
$$
\mathcal B_\perp=\begin{pmatrix}
		0 & I_2\\
		-\check B_\perp & i \check C_\perp
	\end{pmatrix}=\begin{pmatrix}
     0& I_2\\
     -(\eta/\mu)I_2 &0
\end{pmatrix}
$$
	has the eigenvalues $\pm\sqrt{-\eta/\mu}$, which are purely imaginary and semi-simple (of multiplicity two) if and only if $\eta>0$. 
	
	 Next, we see that $\lambda=ib$, $b \in \R$, is an eigenvalue of
	$$\mathcal B_\parallel=\begin{pmatrix}
		0 & I_2\\
		-\check B_\parallel & i\check C_\parallel
	\end{pmatrix}$$
if and only if
\begin{equation}
	\label{pi}
	0=\det(-B_\parallel^{00}b^2-B_\parallel-bC_\parallel)=\kappa\mu b^4-((\tau+\mu)(\nu+\omega)-\kappa\sigma-\nu\mu)b^2+\nu\sigma
:=\kappa\mu \mathfrak p(b^2).
\end{equation}

Clearly, the quadratic $\mathfrak p$ has only positive real roots, i.e.\ 
$\mathcal B_\parallel$ has only purely imaginary eigenvalues, if and only if
\begin{align*}
	 &(\tau+\mu)(\nu+\omega)-\kappa\sigma-\nu\mu \ge 0,\\
	&((\tau+\mu)(\nu+\omega)-\kappa\sigma-\nu\mu)^2-4\nu\mu\kappa\sigma \ge 0.
\end{align*}
If $\nu\sigma>0$, \eqref{hb2} and \eqref{hb3} hold, these eigenvalues are all distinct. 
Next, note that for an eigenvalue $\lambda=ib$ of $\mathcal B_\parallel$ the eigenvectors  
are of the form $(V,ibV)$ with 
$$V \in \ker( B_\parallel^{00}b^2+B_\parallel+bC_\parallel).$$
Thus $\lambda$ is semi-simple of multiplicity $2$  if and only if
\begin{equation} 
	\label{b}
	B_\parallel^{00}b^2+B_\parallel+bC_\parallel=0.
\end{equation}
$\lambda=0$ being a semi-simple eigenvalue of $\mathcal B_\parallel$ implies $\nu=\sigma=0$. 
And in this case the other eigenvalues are purely imaginary  and semi-simple if and only if 
$(\tau+\mu)(\omega+\nu)>0$. Lastly, due to \eqref{b} $\mathcal B_\parallel$ has two non-zero 
semi-simple eigenvalues of multiplicity $2$ if and only if \eqref{C21} holds.
\end{proof}

As mentioned the lemma shows that $B^{\alpha\beta\gamma\delta}$ is hyperbolic if and only 
if it is semi-strictly hyperbolic. In this case $\mathcal B(\omega)$ always has two 
eigenvalues of multiplicity $2$, which in the rest frame correspond to the eigenvalues of 
$\mathcal B_\perp$, and (i), (ii), (iii) above correspond to different multiplicities of the eigenvalues 
of $\mathcal B_\parallel$. In (i) $\mathcal B_\parallel$ all eigenvalues have multiplicity 1, in (ii) 
there are two distinct eigenvalues with multiplicity 2 and in case (iii) $0$ is an eigenvalue of 
multiplicity $2$ and there exist two distinct non-zero eigenvalues of multiplicity $1$. 
As (iii) implies $\mathcal B_\parallel =0$ this is not physical and we will not treat this case any 
further.
 
%We point out that (i) and (ii) are also distinguished by the properties
%also want to distinguish {\color{red} two} different subclasses of (i). Namely 
%\begin{equation*}
%\text{(i)}'~~\sigma, \nu >0,\quad \text{(ii)}'~~ \sigma, \nu <0.
%\end{equation*}
Both cases (i) and (ii)
have been studied before. 
Situation (i) with the further assumption $\nu=\mu$ recovers the equations proposed and 
investigated in \cite{bem19}. 
In situation (ii) \eqref{diffopB} is
second order symmetric hyperbolic in the sense defined in \cite{Hug77}; 
this case comprises the equations proposed in \cite{ft18}, where additionally 
$\mu=\sigma$.\footnote{Note that what is $-\sigma$ here is called $\sigma$ in \cite{ft18}.}

\section{Dissipativity}
It is well-known and also easily checked that under conditions \eqref{ba} the relativistic Euler 
operator 
$$
\psi\mapsto  A^{\alpha\beta\gamma}(\psi)\frac{\partial \psi_\gamma}{\partial x^\delta}
$$ 
is (first-order) symmetric hyperbolic, i.e., the matrices $A^\beta$ are symmetric with $A^0$ 
positive definite. Thus if $B^{\alpha\beta\gamma\delta}$ satisfies (H$_B$), system \eqref{cp1} 
is of hyperbolic-hyperbolic type as introduced in \cite{FS} and the non-linear stability of the 
rest state $\bar \psi=\bar \theta^{-1}(1,0,0,0)^t$, $\bar \theta>0$ fixed, is characterized by the following condition 
(we use 
$\bar B^{\beta\delta}=\check B^{\beta\delta}(\bar \psi),
\bar A^\beta=\check A^\beta(\bar\psi)=(-B^{00})^{-1/2}A^\beta(\bar\psi)(-B^{00})^{-1/2}$ etc.)  
\goodbreak

{\bf Condition (D).}  The tensors $\bar B^{\beta\delta}$, $\bar A^{\beta}$ satisfy three conditions:
\begin{itemize}
	\item[\quad(D1)] 
        There exists a symbolic symmetrizer $S$ for $(A^0)^{-1/2}A(A^0)^{-1/2}$ such that, with 
        $\bar S(\bomega)=S(\bar\psi,\bomega)$, for  
        for every $\bomega\in S^{2}$, all restrictions, as a quadratic form, of
	$$
	W_1=\bar S(\bomega)^{1/2}(\bar A^0)^{-1/2}\big(-\bar B(\bomega)
	+((\bar A^0)^{-1}\bar A(\bomega))^2+\bar C(\bomega)(\bar A^0)^{-1}\bar A(\bomega)\big)
	(\bar A^0)^{-1/2}\bar S(\bomega)^{-1/2}
	$$
	on the eigenspaces $E=J_E^{-1}(\C^n)$ of
	$$
	W_0(\bomega)=\bar S(\bomega)^{1/2}(\bar A^0)^{-1/2}\bar A(\bomega) (\bar A^0)^{-1/2}\bar S(\bomega)^{-1/2}
	$$
	are uniformly negative in the sense that
	$$
	J_E^*\left(
	W_1+W_1^*
	\right)J_E\le -\bar c\
	J_E^*J_E
	\quad\text{with one }\bar c>0.
	%, j=1,\ldots,J.
	$$
	\item[(D2)]
	There exists a symbolic symmetrizer $\mathcal S$ for $i\mathcal B$ such that, with
        $\bar{\mathcal S}(\bomega)=\mathcal S(\bar\psi,\bomega)$ and $\bar{\mathcal B}(\bomega)=\mathcal B(\bar\psi,\bomega)$,    
        for every $\bomega\in S^{2}$, all restrictions, as a quadratic form, of
	$$
	\mathcal W_1=\bar {\mathcal S}(\bomega)^{1/2}
	\begin{pmatrix}
		0&0\\
		-i\bar A(\bomega)&-\bar A^0
	\end{pmatrix}
	\bar{\mathcal S}(\bomega)^{-1/2}
	$$
	on the eigenspaces $\mathcal E=\mathcal J_{\mathcal E}^{-1}(\C^{2n})$ of
	$$
	\mathcal W_0=
	\bar{\mathcal S}(\bomega)^{1/2}
	\bar{\mathcal B}(\bomega)
	\bar{\mathcal S}(\bomega)^{-1/2}
	$$
	are uniformly negative in the sense that
	$$
	\mathcal J_{\mathcal E}^*\left(
	\mathcal W_1+\mathcal W_1^*
	\right)\mathcal J_{\mathcal E}\le -\bar c\ \mathcal J_\mathcal E^*\mathcal J_\mathcal E
	\quad\text{with one }\bar c>0.
	$$
	\item[(D3)] All solutions
	$
	(\lambda,\bxi)\in\C\times(\R^d\setminus\{0\})$
	of the dispersion relation for \eqref{cp1} have
	{\rm Re}$(\lambda)<0$.
\end{itemize}

We note that with respect to 
the orthogonal decomposition $\C^4=(\C \times \bxi \C)\oplus(\{0\} \times \{\bxi\}^{\perp})$
also $A^0$ and $A(\bxi)$
decompose in the sense of linear operators as 
$$A^0=A^0_\parallel \oplus A^0_\perp,\qquad A(\bxi)=A_\parallel(\xi)\oplus A_\perp(\xi).
$$
So we can treat also the dissipativity on $\{0\} \times \{\bxi\}^\perp$ and $\C \times \C\bxi$ 
separately.

For the case mentioned in the last sentence of the previous section,
non-linear stability of the rest state was first  shown in \cite{sro18} --
a generalization for all systems that satisfy (ii) is straightforward:

\begin{lemma}
	\label{ft}
	Assume $\kappa,\mu,\eta>0$ and \eqref{C21}. Then (D1), (D2) and (D3) are 
all equivalent to \eqref{C22}.
\end{lemma}

\begin{proof}
	
	 By the arguments presented in \cite{FS}, Section 3.2, (D) holds in general for 
         $-B^{00}, B(\bomega), A^0>0$ and $A(\bomega)=C(\bomega)=0$. 
         Thus it holds here on $\{0\} \times \{\bxi\}^\perp$.
	
	To study it on $\C \times \C\bxi$, set $\tilde \sigma:=-\sigma/\mu$ and note that 
        \eqref{C21}implies $B_\parallel=\tilde \sigma B_\parallel^{00}$, i.e.,
        $\bar B_\parallel=\tilde \sigma I_2$.
	As $\tilde \sigma>0$ conditions  (D1), (D2), (D3) are all equivalent to
	$$\tilde \sigma^{1/2}\bar A^0_\parallel\pm \bar A_\parallel(\xi)>0$$
	(cf. \cite{FS}, Section 3.2),
	which directly yields the assertion.
\end{proof}
%In the following, we use the abbreviation $r=c_s^{-2}$.
\begin{lemma}
	\label{dissip}
	 Assume $\kappa,\mu,\eta>0$, $\nu\sigma>0$ and \eqref{hb2}, \eqref{hb3}. 
         Then the following hold
	 \begin{enumerate}
	 	\item[(i)] (D3) is equivalent to
	 \begin{equation}
	 	\label{D311}
	 	c_s^{-2}(\omega+\tau
	 	)-\kappa-c_s^{-4}\sigma \ge 0,
	 \end{equation}
 	\begin{equation}
 		\label{D321}
 	\begin{aligned}
 		(\kappa+c_s^{-2}\mu)(\tau+\omega+\mu-c_s^{-2}\sigma)[(\tau+\mu)(\omega+\nu)-(\mu\nu+\kappa \sigma)]&\\
 		-\kappa \mu(\tau+\omega+\mu-c_s^{-2}\eta)^2-(\kappa+c_s^{-2}\mu)^2\nu\sigma&\ge0.
 	\end{aligned}
 	\end{equation}
  	where at least one of the inequalities is strict. 
  	\item[(ii)]
  	(D1) is equivalent to \eqref{D31}.
  	\item[(iii)]
  	Assume (D3). Then (D2) is equivalent to
  	\begin{equation}
  		\label{D2}
  		\begin{aligned}
  		(\kappa+c_s^{-2}\mu)(\tau+\omega+\mu-c_s^{-2}\sigma)[(\tau+\mu)(\omega+\nu)-(\mu\nu+\kappa \sigma)]&\\
  		-\kappa \mu(\tau+\omega+\mu-c_s^{-2}\eta)^2-(\kappa+c_s^{-2}\mu)^2\nu\sigma&\neq 0.
  		\end{aligned}
  	\end{equation}
  	\end{enumerate}
\end{lemma}

\begin{proof}
	As mentioned in the last proof, (D1), (D2), (D3) are satisfied on 
        $\{0\} \times \{\bxi\}^\perp$ under the assumptions $\mu, \eta>0$. 
        Thus we only need to consider the symbols on $\C \times \C\bxi$. All matrices below are evaluated at $\psi=\bar\psi$ and
        we drop the subscript $\parallel$. 
	
(i) The dispersion relation is given as
	\begin{align*} 
		0&=\det(-\lambda^2B^{00}+B(\xi)-i\lambda C(\xi)+\lambda A^0+iA(\xi))\\
		&=\kappa \mu\lambda^4+(\kappa+c_s^{-2}\mu)\lambda^3+(\xi^2((\tau+\mu)(\nu+\omega)-\kappa\sigma-\nu\mu)
                +c_s^{-2})\lambda^2+\lambda\xi^2(\tau+\omega+\mu-c_s^{-2}\sigma)\\
		&\quad +\xi^4\nu\sigma+\xi^2.
	\end{align*}
 We directly see that by rescaling $\kappa$ by $c_s^{-4}$, $\mu,\nu,\omega,\tau$ by $c_s^{-2}$ and $\xi^2,\lambda$ by 
$c_s^2$, the dispersion relation and \eqref{D31}, \eqref{D32} become independent of $c_s$. It is thus w.l.o.g.\ that we assume
for the rest of this proof that $c_s=1$. We use $\alpha=\xi^2$ and set
 \begin{align*}
 	a_0&=\alpha^2\nu\sigma+\alpha,~~a_1=\alpha(\tau+\omega+\mu-\sigma),\\
 	a_2&=\alpha((\tau+\mu)(\nu+\omega)-\kappa\sigma-\nu\mu)+1,~~a_3=\kappa+\mu,~~a_4=\kappa \mu.
 \end{align*}
 By the Routh-Hurwitz criterion, (D3) is equivalent to $a_i > 0$, $i=0,\ldots 4$, and 
 $$\Delta(\alpha):=a_1a_2a_3-(a_1^2a_4+a_3^2a_0)>0$$
 for all $\alpha \in (0,\infty)$. 
 Clearly $a_0,a_3,a_4>0$ by assumption and $a_2>0$ by \eqref{hb2}. Furthermore
 $$\Delta(\alpha)=\alpha(\Delta^{(1)}+\alpha\Delta^2),$$
 where
 \begin{align*} 
 	\Delta^{(1)}&=\omega+\tau-\kappa-\sigma\\
 	\Delta^{(2)}&=(\kappa+\mu)(\tau+\omega+\mu- \eta)[(\tau+\mu)(\omega+\nu)-(\mu\nu+\kappa \sigma)]-\kappa \mu(\tau+\omega+\mu-\sigma)^2-(\kappa+\mu)^2\nu\sigma.
 \end{align*}
As $\Delta^{(1)} \ge 0$ implies $a_1 >0$ this finishes the proof of (i).

(ii) To show (ii) and (iii) 
note that 
$$\bar B^{00}=-I_2,\bar B=\bar B(1)=-\begin{pmatrix}
	\frac{\nu}{\kappa} & 0\\
	0 & \frac{\sigma}{\mu}
\end{pmatrix},~~\bar C=\bar C(1)=-\frac1{\sqrt{\kappa\mu}}\begin{pmatrix} 0 & \tau+\nu\\ \omega+\mu \end{pmatrix},$$
$$\bar A^0=\begin{pmatrix} \frac1{\kappa} & 0 \\ 0 & \frac{1}{\mu} \end{pmatrix},\quad \bar A=\bar A(1)=\frac{1}{\sqrt{\kappa\mu}}\begin{pmatrix}0 &1 \\ 1&0\end{pmatrix}.$$

The eigenspaces of
$$W_0=(\bar A^0)^{-1/2}\bar A(\bar A^0)^{-1/2}=\begin{pmatrix} 0 & 1 \\ 1 &0
\end{pmatrix}$$
are $E_{\pm}=\C(\pm 1,1)^t$ and the restriction of 
$$W_1=(\bar A^0)^{-1/2}(-\bar B+(\bar A^0)^{-1}\bar A(\bar A^0)^{-1}\bar A+\bar C(\bar A^0)^{-1}\bar A)(\bar A^0)^{-1/2}$$
on $E_\pm$ is
$$\sigma+\kappa-(\omega+\tau),$$
which proves (ii).
\goodbreak

(iii) %{\color{blue} 
For better readability we only give the proof for $(\tau+\mu)(\omega+\nu)\neq 0$ at this point. 
The other case follows with analogous arguments.
%} 
By (C1.1), (C1.2) the matrix
$$\bar{\mathcal B}=\begin{pmatrix}
	0 & I\\
	-\bar B &i\bar C \end{pmatrix}$$
has four simple purely imaginary eigenvalues, which are given by
$$\lambda_{\pm s}=\pm ib_s=\pm\sqrt{-\beta_s},~~s=1,2$$
where $0<\beta_1<\beta_2$ are the roots of the polynomial $\mathfrak p$ that was defined in \eqref{pi}. To determine the left and right eigenvectors write
$$\bar{\mathcal B}
=\begin{pmatrix}
0 & 0 & 1 &0\\
0 & 0&0 &1\\
\mathfrak b & 0 &0 &-i\mathfrak a\\
0 & \mathfrak c &-i\mathfrak d &0
\end{pmatrix},$$
where
$$\mathfrak a=\frac{\tau+\mu}{\sqrt{\kappa\mu}},~~\mathfrak b=\frac{\nu}{\kappa},~~\mathfrak c=\frac{\sigma}{\mu},~~\mathfrak d
=\frac{\omega+\nu}{\sqrt{\kappa\mu}}.$$
Since $(\tau+\mu)(\omega+\nu) \neq 0$ we find that for $s=-2,-1,1,2$ a left respectively  right eigenvector corresponding to 
$\lambda_s$ is given as
$$L_s=\begin{pmatrix}
	-\mathfrak d\mathfrak bb_s , (\mathfrak d\mathfrak a-\mathfrak b-b_s^2)b_s^2,-i\mathfrak d, i(b_s^2+\mathfrak b)b_s
\end{pmatrix},\quad R_s=\begin{pmatrix}
	\mathfrak ab_s\\ -(\mathfrak b+b_s^2)\\i\mathfrak ab_s^2\\-i(\mathfrak b+b_s^2)b_s
\end{pmatrix}.$$
Assuming (D3) condition (D2) is now equivalent to
$$L_s
\begin{pmatrix} 
0&0\\-i\bar A&-\bar A^0
\end{pmatrix}
R_s \neq 0, s=-2,-1,1,2.$$
We find for $s=1,2$
\begin{align*}L_{\pm s} \mathcal A R_{\pm_s}
=
\frac{\beta_s}{\mu}\left[-\beta_s^2+\frac1\kappa\left(\tau+\omega+\mu-\nu-\frac1{\kappa}(\tau+\nu)(\omega+\mu)\right)\beta_s+\frac{(\tau+\omega+\mu)\nu}{\kappa^2} \right]=:\frac{\beta_s}{\mu}q(\beta_s).\end{align*}
With 
\begin{equation}
	\label{kl}
	 \mathfrak k:=\frac{(\tau+\mu)(\nu+\omega)-\kappa\sigma-\nu\mu}{\kappa\mu},\quad \mathfrak l:=\frac{\nu\sigma}{\kappa\mu}
\end{equation}
the zeros of the polynomial $\mathfrak p$ are given as
\begin{align*}\beta_1&=\frac12\left(\mathfrak k-\sqrt{\mathfrak k^2-4\mathfrak l}\right),\quad \beta_2=\frac12\left(\mathfrak k+\sqrt{\mathfrak k^2-4\mathfrak l}\right).
\end{align*} 
Also using
$$\mathfrak m=\frac1\kappa\left(\tau+\omega+\mu-\nu-\frac1{\kappa}(\tau+\nu)(\omega+\mu)\right),\quad \mathfrak n=\frac{(\tau+\omega+\mu)\nu}{\kappa^2}$$
we get that $q(\beta_s) \neq 0$ is equivalent to
$$\pm (2(\mathfrak l+\mathfrak n)+\mathfrak k(\mathfrak m-\mathfrak k)) \neq (\mathfrak k-\mathfrak m)\sqrt{\mathfrak k^2-4\mathfrak l}.$$
Squaring both sides gives
$$(\mathfrak l+\mathfrak n)^2+(\mathfrak m-\mathfrak k)(\mathfrak m\mathfrak l+\mathfrak k\mathfrak n) \neq 0$$
and straightforward calculations give
\begin{align*}
	&\frac{\kappa^5\mu^2}{\nu(\tau+\nu)(\omega+\mu)}((\mathfrak l+\mathfrak n)^2+(\mathfrak m-\mathfrak k)(\mathfrak m\mathfrak l+\mathfrak k\mathfrak n))\\
	&=(\kappa+\mu)(\tau+\omega+\mu- \eta)[(\tau+\mu)(\omega+\nu)-(\mu\nu+\kappa \sigma)]-\kappa \mu(\tau+\omega+\mu-\sigma)^2-(\kappa+\mu)^2\nu\sigma,
\end{align*}
which shows the assertion.
\end{proof}

\section{Causality}
\setcounter{equation}0
Equations describing the dynamics of fluids in the relativistic regime need to be causal, i.e. information must not propagate faster than the speed of light. 
This did not play a role in the previous argumentation as hyperbolicity and dissipation can be achieved 
independently of causality. However, at this point we will also formulate conditions equivalent to 
causality in order to characterize the regime of parameters for which the equations are 
physically relevant. 

\begin{defi}
	The hyperbolic differential operator
	\begin{equation*}
		\psi\mapsto
		B^{\alpha\beta\gamma\delta}(\psi)\frac{\partial\psi_\gamma}{\partial x^\beta\partial x^\delta}
	\end{equation*}
is causal if at any $\psi\in\mathcal U$ all solutions 
$(\varphi_0,\varphi_1,\varphi_2,\varphi_3) \in \R^4\setminus\{0\}$ of
$$\det(B^{\alpha\beta\gamma\delta}(\psi)\varphi_\beta\varphi_\delta)=0$$
are spacelike. 
\end{defi}

Note that causality is a Lorentz invariant property and thus again we only need to study the rest-frame representation 
of $B^{\alpha\beta\gamma\delta}$.

\begin{lemma}
	Let the operator \eqref{diffopB} be hyperbolic. Then it is causal if and only if $\eta/\mu \le 1$ 
        and
	\begin{align}
	%	\label{ca1}
	%	   &(\tau+\mu)(\nu+\omega)-\kappa\sigma-\nu\mu \ge 0,\\
	%	   \label{ca2}
	%	  &((\tau+\mu)(\nu+\omega)-\kappa\sigma-\nu\mu)^2-4\nu\mu\kappa\sigma \ge 0,\\
		  \label{ca4}
		  &(\tau+\mu)(\nu+\omega)-\kappa\sigma-\nu\mu \le \kappa \mu+\nu\sigma\quad\text{as well as}\quad \nu\sigma\le \kappa\mu.
	\end{align}
	
\end{lemma}

\begin{proof}
	It is obvious that 
	$$\det(B^{\alpha\beta\gamma\delta}(\bar \psi)\varphi_\beta\varphi_\delta)=0$$
    is equivalent to the fact that $\varphi_0$ is a characteristic speed of 
    $i\mathcal B(\bar \psi, \varphi_1,\varphi_2,\varphi_3)$
    As seen in the proof of Lemma \ref{hb} these speeds
    are, in the rest frame,
    $$b^\bot_\pm=\pm\sqrt{\frac{\eta}{\mu}},\quad b_{\pm s}=\pm \sqrt \beta_s,~~s=1,2$$
    with $\beta_1,\beta_2$ 
%    \begin{align*}
%    	\beta_1&=\frac12\left(\mathfrak k-\sqrt{\mathfrak k^2-4\mathfrak l}\right),\quad \beta_2=\frac12\left(\mathfrak k+\sqrt{\mathfrak k^2-4\mathfrak l}\right)\\
%    	%\mathfrak k&=\frac{(\tau+\mu)(\nu+\omega)-\kappa\sigma-\nu\mu}{\kappa\mu},\quad \mathfrak l=\frac{\nu\sigma}{\kappa\mu}.
%    \end{align*}
%    Thus they are real if and only if $\eta \ge 0$ and \eqref{ca1}, \eqref{ca2} hold (of course, this was also already shown in the proof of Lemma \ref{hb}).
    the two roots of the polynomial
    $$
    \mathfrak p(\beta)=\beta^2-\mathfrak k\beta +\mathfrak l
    $$
    as above.
    Trivially, $|b^\bot_\pm|\le 1$ if and only if $\eta \le \mu$,  and 
    the property $|\beta_s| \le 1$, $s=1,2$,
    is equivalent to $\mathfrak k-1\le \mathfrak l\le \mathfrak l$, which is \eqref{ca4}.
    \end{proof}
\goodbreak
We finally look at the causality of the left-hand side of \eqref{cp1}, assuming that 
\begin{equation}\label{symm}
       C=C^\top. 
\end{equation}

\begin{lemma}  
If (D) holds and $b_{max}$ is the largest and  
$
b_{min}=-b_{max}
$
the smallest eigenvalue of $i\bar{\mathcal B}$,
the eigenvalues $a$ of $(\bar A^0)^{-1/2}\bar A(\bomega) (\bar A^0)^{-1/2}$ satisfy 
\begin{equation}\label{subchar}
b_{min}\le a\le b_{max}.
\end{equation}
In particular, if the operator \eqref{diffopB} is causal, then subluminality of the speed of sound, $c_s^2\le 1$, is a necessary   
condition for system \eqref{cp1} to have the dissipativity property (D). 
\end{lemma}
\begin{proof}
(D1) is violated if $W_1+W_1^*$, with $S=I$, is positive on an eigenspace of $(\bar A^0)^{-1/2}\bar A(\bomega) (\bar A^0)^{-1/2}$. 
However, restricted to such eigenspace associated with an eigenvalue $a$,
$$
W_1=W_1^*=(\bar A^0)^{-1/2}(-\bar B+a\bar C+a^2I)(\bar A^0)^{-1/2}
$$
and as 
\begin{equation}\label{W1pos}
\det(-\bar B+a\bar C+a^2I) = \det(i\bar{\mathcal B}-aI),
 %\mathfrak{q}(a^2):=\kappa\mu a^4-((\tau+\mu+\omega+\nu)^2/4-(\kappa\sigma+\nu\mu))a^2+\nu\sigma.
\end{equation}
$W_1+W_1^*$ is positive if $a\notin [b_{min},b_{max}]$. In particular, $c_s^2\le 1$ if $b_{max}\le 1$.
\end{proof}
\begin{rem} 
Inequality \eqref{subchar} means a classical \emph{subcharacteristic} condition, saying that the speed range 
of the equilibrium system must be contained in that of the regularizing operator (cf.\ \cite{w,l,cll,y}).
\end{rem}
\begin{rem}
Properties \eqref{BBC} show that 
assumption \eqref{symm} is satisfied both in the positive case ((C2), $B_{\parallel}>0$), 
thus comprising the description advocated in \cite{ft14,ft18}, and in 
the negative case ((C1), $B_{\parallel}<0$) if 
$$
\mu=\nu\quad\text{and}\quad\tau=\omega;  
$$
this latter condition characterizes the fully symmetric setting that was introduced in 
Definition 2(i) and eq.\ (2.7) of \cite{f20}, here
comprising in particular the original model of \cite{bem18}.
\end{rem}

{\bf Acknowledgment.} This research was supported by DFG Grant No.\ FR 822/11-1 and in part by grant NSF PHY-1748958 to the Kavli Institute for Theoretical Physics (KITP). H. F. thanks the Institute and the 
organizers of its programme "The Many Faces of Relativistic Fluid Dynamics" in which 
he participated during May/June 2023.

{\bf Conflict of interest.} The authors have no competing interests to declare.

\end{document}